\documentclass[a4paper,12pt]{amsart}
\usepackage{amsmath}
\usepackage{amssymb}
\usepackage{mathrsfs}
\usepackage{enumerate}
\usepackage{ifthen}
\usepackage{graphicx}
\usepackage[T1]{fontenc} 
\usepackage{tabularx}
\usepackage{multirow}
\usepackage{color,soul}

\everymath{\displaystyle}

\nonstopmode \numberwithin{equation}{section}
\newtheorem{thm}{Theorem}[section]
\newtheorem{cor}{Corollary}[section]
\newtheorem{lem}{Lemma}[section]

\theoremstyle{definition}

\newcounter{minutes}
\divide\time by 60
\newcounter{hours}
\multiply\time by 60
\addtocounter{minutes}{-\time}

\newcounter {own}
\def\theown {\thesection       .\arabic{own}}

{\qed\bigskip}

\newcounter{alphabet}

\begin{document}

\title{Pre-Schwarzian norm estimate for certain Ma-Minda Class of functions}

\author{Md Firoz Ali}
\address{Md Firoz Ali,
 National Institute Of Technology Durgapur,
West Bengal-713209, India}
\email{ali.firoz89@gmail.com, fali.maths@nitdgp.ac.in}

\author{Md Nurezzaman}
\address{Md Nurezzaman, National Institute Of Technology Durgapur,
West Bengal-713209, India}
\email{nurezzaman94@gmail.com}

\author{Sanjit Pal}
\address{Sanjit Pal,
 National Institute Of Technology Durgapur,
West Bengal-713209, India}
\email{palsanjit6@gmail.com}

\subjclass[2020]{Primary 30C45, 30C55}
\keywords{analytic functions, convex functions, starlike functions, Ma-Minda class of starlike functions, pre-Schwarzian norm.}

\def\thefootnote{}
\footnotetext{ {\tiny File:~\jobname.tex,
printed: \number\year-\number\month-\number\day,
          \thehours.\ifnum\theminutes<10{0}\fi\theminutes }
} \makeatletter\def\thefootnote{\@arabic\c@footnote}\makeatother

\begin{abstract}
Let $\mathcal{S}^*(\varphi)$ be the class of all analytic functions $f$ in the unit disk $\mathbb{D}=\{z\in\mathbb{C}:|z|<1\}$, normalized by $f(0)=f'(0)-1=0$ that satisfy the subordination relation $zf'(z)/f(z)\prec\varphi(z)$, where $\varphi$ is an analytic and univalent  in $\mathbb{D}$ with ${\rm Re\,}\varphi(z)>0$ such that $\varphi(\mathbb{D})$ is symmetric with respect to the real axis and stralike with respect to $1$. In the present article, we obtain the sharp estimates of the pre-Schwarzian norm of $f$ and the Alexander transformation $J[f]$ for functions $f(z)$ in the class $\mathcal{S}^*(\varphi)$ when $\varphi(z)=e^{\lambda z}$, $0<\lambda\le\pi/2$ and $\varphi(z)=\sqrt{1+cz}$, $0<c\le1.$
\end{abstract}

\thanks{}

\maketitle
\pagestyle{myheadings}
\markboth{Md Firoz Ali, Md Nurezzaman and Sanjit Pal}{Pre-Schwarzian norm for certain Ma-Minda class}
\section{Introduction}
Let $\mathcal{H}$ be the class of all analytic functions in the open unit disk $\mathbb{D}:=\{z\in\mathbb{C}:|z|<1\}$ and let $\mathcal{A}$ be the subclass of $\mathcal{H}$ with $f(0)=0=f'(0)-1$. Thus any $f\in\mathcal{A}$ has the following Taylor series form
\begin{equation}\label{R-01}
f(z)= z+\sum_{n=2}^{\infty}a_n z^n.
\end{equation}
Further, let $\mathcal{S}$ be the subclass  of $\mathcal{A}$ that are  univalent (that is, one-to-one) in $\mathbb{D}$. A function $f\in\mathcal{A}$ is called starlike if $f(\mathbb{D})$ is a starlike domain with respect to the origin, i.e., for $w_0\in f(\mathbb{D})$, the line segment joining $0$ and $w_0$ is also lies in $f(\mathbb{D})$.
The set of all starlike functions in $\mathcal{S}$ is denoted by $\mathcal{S}^*$. It is well-known that a function $f\in\mathcal{A}$ is in $\mathcal{S}^*$ if and only if ${\rm Re\,} (zf'(z)/f(z))>0$ for $z\in\mathbb{D}$. Analogously, a function $f\in\mathcal{A}$ is called convex if $f(\mathbb{D})$ is a convex domain, i.e., $f(\mathbb{D})$ is a starlike domain with respect to each of its points.
The set of all convex functions in $\mathcal{S}$ is denoted by $\mathcal{C}$. It is well-known that a function $f\in\mathcal{A}$ is in $\mathcal{C}$ if and only if ${\rm Re\,} \left(1+zf''(z)/f'(z)\right)>0$ for $z\in\mathbb{D}$. For more details about these classes, we refer to \cite{1983-Duren, 1983-Goodman}.\\

For two functions $f$ and $g$ in $\mathcal{H}$, we say that $f$ is subordinate to $g$, written as $f\prec g$ if there exists a function $\omega\in\mathcal{H}$ with $\omega(0)=0$ and $|\omega(z)|<1$ such that $f(z)=g(\omega(z))$ for $z\in\mathbb{D}$. Let $\varphi$ be an analytic univalent function with positive real part in $\mathbb{D}$ such that $\varphi(\mathbb{D})$ is symmetric with respect to the real axis and starlike with respect to $\varphi(0) = 1$ and $\varphi'(0)>0$.
For such $\varphi$, Ma and Minda \cite{1992-Ma-Minda} introduced the classes $\mathcal{S}^*(\varphi)$ and $\mathcal{C}(\varphi)$ as
\begin{align*}
\mathcal{S}^*(\varphi)=\left\{f\in\mathcal{A}:\frac{zf'(z)}{f(z)}\prec\varphi(z)\right\},
\end{align*}
and
\begin{align*}
\mathcal{C}(\varphi)=\left\{f\in\mathcal{A}:1+\frac{zf''(z)}{f'(z)}\prec\varphi(z)\right\},
\end{align*}
respectively. Sometimes $\mathcal{S}^*(\varphi)$ and $\mathcal{C}(\varphi)$ are called Ma-Minda classes of starlike and convex functions, respectively. One can easily prove the inclusion relations $\mathcal{S}^*(\varphi)\subset\mathcal{S}^*$ and $\mathcal{C}(\varphi)\subset\mathcal{C}$. It is important to note that $f\in\mathcal{S}^*(\varphi)$ if and only if $J[f]\in\mathcal{C}(\varphi)$, where $J[f]$ is the Alexander transformation of $f$ defined by 
$$
J[f](z)=\int_0^z\frac{f(t)}{t}dt=f(z)*(-\log(1-z)).
$$
The symbol $*$ stands for the usual Hadamard product (or convolution).\\

For different choices of the function $\varphi$, we get many well known geometric subclasses of $\mathcal{A}$. For example, if we take $\varphi(z)=(1+z)/(1-z)$ then the classes $\mathcal{S}^*(\varphi)$ and $\mathcal{C}(\varphi)$ reduces to the classes $\mathcal{S}^*$ of starlike functions and $\mathcal{C}$ of convex functions. For $\varphi(z)=(1+(1-2\alpha)z)/(1-z),~0\le \alpha<1$, one may get the classes $\mathcal{S}^*(\alpha)$ of starlike function of order $\alpha$ and $\mathcal{C}(\alpha)$ of convex function of order $\alpha$. For $\varphi=((1+z)/(1-z))^\gamma,~0< \gamma\le1$, Stankiewicz \cite{1971-Stankiewicz} introduced the classes $\mathcal{S}\mathcal{S}^*(\gamma)=\mathcal{S}^*(\varphi)$ and $\mathcal{S}\mathcal{C}(\gamma)=\mathcal{C}(\varphi)$ which are known as the class of strongly starlike function of order $\gamma$ and   strongly convex function of order $\gamma$. Also for  $\varphi=(1+Az)/(1+Bz),~-1\le B<A\le 1$, we have the classes of Janowski starlike and convex functions $\mathcal{S}^*(A,B)$ and $\mathcal{C}(A,B)$, respectively (see \cite{1973-Janowski}).
Further, the class $\mathcal{S}^*(\varphi)$ with $\varphi(z)=(1+2/\pi^2(\log(1-\sqrt{z})/(1+\sqrt{z}))^2)$ was introduced and studied by R\o nning \cite{1993-Ronning}. Raina and Sok\'{o}\l~\cite{2015-Raina} introduced and studied the class $\mathcal{S}^*(\varphi)$ with $\varphi(z)=z+\sqrt{1+z^2}$. Kumar et al. \cite{2019-Kumar} introduced the family $\mathcal{S}^*(\varphi)$ with $\varphi(z)=1+\sin z$ and also obtained radius of starlikeness and radius of convexity. The class $\mathcal{S}^*(\varphi)$ with $\varphi(z)=2/(1+e^{-z})$ was introduced by Goel and Kumar \cite{2020-Goel} and studied radius problems, coefficient bounds, growth and distortion theorem and inclusion relations.\\

In the present article, we concentrate on two different classes of functions, namely, $\mathcal{S}^*_{\lambda e}=\mathcal{S}^*(\varphi)$ with $\varphi=e^{\lambda z},~0<\lambda\le \pi/2$ and $\mathcal{S}^*(q_c)=\mathcal{S}^*(\varphi)$ with $\varphi=\sqrt{1+cz},~0<c\le 1$. More precisely,
\begin{align}\label{R-05}
\mathcal{S}^*_{\lambda e}=\left\{f\in\mathcal{A}
:\frac{zf'(z)}{f(z)}\prec e^{\lambda z}\right\}
\end{align}
and
\begin{align}\label{R-10}
\mathcal{S}^*(q_c)=\left\{f\in\mathcal{A}
:\frac{zf'(z)}{f(z)}\prec \sqrt{1+cz}\right\}.
\end{align}

For $\lambda=1$, the class $\mathcal{S}^*_{\lambda e}=\mathcal{S}^*_{e}=\mathcal{S}^*(e^z)$ was introduced and studied by Mendiratta et al. \cite{2014-Mendiratta}. Later the family $\mathcal{S}^*_{\lambda e}$ was introduced and studied by Wang et al. \cite{2020-L.Shi}.  To see more results on this class, we refer to
\cite{2018-Kumar-Ravichandran,2017-KUMAR-RAVICHANDRAN-VERMA}. Likewise, for $c=1$, the class $\mathcal{S}^*(q_c)$ reduces to the class $\mathcal{S}^*(\sqrt{1+z})$ which was introduced and studied by Sok\'{o}\l~ and Stankiewicz \cite{1996-Sokol}. The family $\mathcal{S}^*(q_c)$ was introduced and studied by Aouf et al. \cite{2011-Aouf-Dziok-Sokol}. It is important to note that $f\in\mathcal{S}^*_{\lambda e}(\text{respectively},~\mathcal{S}^*(q_c))$ if and only if $J[f]\in \mathcal{C}_{\lambda e}\left(\text{respectively},~\mathcal{C}(q_c)\right)$, where  the classes $\mathcal{C}_{\lambda e}$ and $\mathcal{C}(q_c)$ are defined by
\begin{align}\label{R-12}
\mathcal{C}_{\lambda e}=\left\{f\in\mathcal{A}:1+\frac{zf''(z)}{f'(z)}\prec e^{\lambda z}\right\}
\end{align}
and
\begin{align}\label{R-14}
\mathcal{C}(q_c)=\left\{f\in\mathcal{A}:1+\frac{zf''(z)}{f'(z)}\prec \sqrt{1+cz}\right\}.
\end{align}
For $\lambda=1$, the class $\mathcal{C}_{\lambda e}=\mathcal{C}_{e}=\mathcal{C}(e^z)$ was introduced and studied by Mendiratta et al. \cite{2014-Mendiratta}. Later the class $\mathcal{C}_{\lambda e}$ was introduced and studied by Wang et al. \cite{2020-L.Shi}. \\

 It is easy to see that  for function $f$ in $\mathcal{S}^*_{\lambda e}$ satisfy the condition
 $$
 \left|\log\frac{zf'(z)}{f(z)}\right|\le \lambda,\quad z\in\mathbb{D},
 $$
 whereas for any $f$ in $\mathcal{S}^*(q_c)$ one can easily verify that
$$
\left|\left(\frac{zf'(z)}{f(z)}\right)^2-1\right|\le c, \quad z\in\mathbb{D}.
$$

It is clear that a function $f\in\mathcal{S}^*_{\lambda e}$ if and only if there exists a $p_1\in\mathcal{H}$ with $p_1\prec e^{\lambda z},~0<\lambda\le \pi/2$  such that
\begin{align}\label{R-15}
f(z)=z~\mathrm{exp}{\int_0^z\frac{p_1(t)-1}{t}dt} .
\end{align}
Similarly, a function $f\in\mathcal{S}^*(q_c)$ if and only if there exists a $p_2\in\mathcal{H}$ with $p_2\prec \sqrt{1+cz},~0<c\le1$ such that
\begin{align}\label{R-20}
f(z)=z~\mathrm{exp}{\int_0^z\frac{p_2(t)-1}{t}dt}.
\end{align}
Particularly, if we choose $p_1(z)=e^{\lambda z}$ and $p_2(z)=\sqrt{1+cz}$ in \eqref{R-15} and \eqref{R-20}, respectively, we obtain
\begin{align*}
f_1(z)=z~\mathrm{exp}{\int_0^z\frac{e^{\lambda t}-1}{t}dt}\in \mathcal{S}^*_{\lambda e},
\end{align*}
and
\begin{align*}
f_2(z)=z~\mathrm{exp}{\int_0^z\frac{\sqrt{1+ct}-1}{t}dt}\in \mathcal{S}^*(q_c).
\end{align*}
The functions $f_1(z)$ and $f_2(z)$ plays the role of extremal function for many extremal problems in the classes $\mathcal{S}^*_{\lambda e}$ and $\mathcal{S}^*(q_c)$, respectively.\\

\section{Pre-Schwarzian Norm}

Let $\mathcal{LU}$ denote the subclass of $\mathcal{H}$ consisting of all locally univalent functions in $\mathbb{D}$, i.e., $\mathcal{LU}:=\{f\in\mathcal{H}:f'(z)\ne 0\text{ for all }z\in\mathbb{D}\}$. For a locally univalent function $f\in\mathcal{LU}$, the pre-Schwarzian derivative is defined by
$$
P_f(z):=\frac{f''(z)}{f'(z)},
$$
and the pre-Schwarzian norm (the hyperbolic sup-norm) is defined by
$$
||P_f||:=\sup\limits_{z\in\mathbb{D}}(1-|z|^2)|P_f(z)|.
$$
This norm has a significant meaning in the theory of Teichm\"{u}ller spaces. For a univalent function $f$, it is well known that $||P_f||\leq 6$. On the other hand, if $||P_f||\leq 1$ then the function $f$ is univalent in $\mathbb{D}$. Both the constants $6$ and $1$ are best possible (see \cite{1972-Becker, 1984-Becker-Pommerenke-1984}). In 1976, Yamashita \cite{1976-Yamashita} proved that $||P_f ||$ is finite if and only if $f$ is uniformly locally univalent in $\mathbb{D}$. Moreover, if $||P_f||<2$, then $f$ is bounded in $\mathbb{D}$ (see \cite{2002-Kim-Sugawa}).\\

In univalent function theory, several researchers determined the pre-Schwarzian norm for different subclasses of analytic and univalent functions.
In 1998, Sugawa \cite{1998-Sugawa} obtained sharp estimate of the pre-Schwarzian  norm for functions in the class $S^*(\varphi)$ with $\varphi=\left((1+z)/(1-z)\right)^\alpha$ of strongly starlike functions of order $\alpha$, $0<\alpha\le 1$.
 In $1999$, Yamashita \cite{1999-Yamashita} studied the classes $\mathcal{S}^*(\alpha)$ and $\mathcal{C}(\alpha)$, $0\le \alpha<1$, and proved the sharp estimates $||P_f||\le6-4\alpha$ for $f\in\mathcal{S}^*(\alpha)$ and $||P_f||\le 4(1-\alpha)$ for $f\in\mathcal{C}(\alpha)$. A function $f\in\mathcal{A}$ is said to be $\alpha$-spirallike function if ${\rm Re}(e^{-i\alpha}zf'(z)/f(z))>0$ for $z\in\mathbb{D}$, where $-\pi/2<\alpha<\pi/2$. In $2000$, Okuyama \cite{2000-Okuyama} obtained the sharp estimate of the pre-Schwarzian norm for $\alpha$-spirallike functions. Kim and Sugawa \cite{2006-Kim-Sugawa} studied the class $\mathcal{C}(A,B)$ and obtained the sharp estimate of the pre-Schwarzian norm $||P_f||\le2(A-B)/(1+\sqrt{1-B^2})$ (see also \cite{2008-Ponnusamy-Sahoo}).  Ponnuswamy and Sahoo \cite{2010-Ponnusamy} obtained the sharp estimates of the pre-Schwarzian norm for functions in the class
$
\mathcal{S}^*[\alpha,\beta]=S^*(\varphi),~\varphi(z)= \left((1+(1-2\beta)z)/(1-z)\right)^\alpha,
$
where $0<\alpha\le1$ and $0\le\beta<1$. In $2014$, Aghalary and Orouji \cite{2014-Aghalary} obtained the sharp estimate of the pre-Schwarzian norm for $\alpha$-spirallike function of order $\rho$, where $-\pi/2<\alpha<\pi/2$ and $0\le\rho<1$. \\

The pre-Schwarzian norm of certain integral transform of $f$ (e.g. Alexander transform $J[f]$) for certain subclass of $f$ has been also studied in the literature. One of the general integral transformation is the Bernardi integral transformation, which plays a significant role to connect with other integral transformations in geometric function theory.  For any $\gamma>-1$ and $f\in\mathcal{A}$, the Bernardi integral transformation $B_\gamma[f]$ is defined by 
$$
B_\gamma[f](z)=\frac{\gamma+1}{z^\gamma}\int_0^zt^{\gamma-1}f(t)dt=f(z)*zF(1,\gamma+1;\gamma+2;z),
$$
where $F(a,b;c;z)$ is the well-known Gaussian hypergeometric function. It is clear that $B_0[f]=J[f]$. Here the univalence of $f$ does not necessarily implies the univalence of $J[f]$ (see \cite{1983-Duren}). In $2004$, Kim et al. \cite{2004-Kim-Ponnusamy-Sugawa} obtained the sharp estimate $||P_{J[f]}||\le4$ of the Alexander transformation $J[f]$ for functions $f$ in $\mathcal{S}$. Parvatham et al. \cite{2008-Parvatham-Ponnusamy-Sahoo} obtained the sharp estimate of the pre-Schwarzian norm of the Bernardi integral transform $B_\gamma[f], \gamma>-1$ for functions
in $C(A, B)$. Ponnusamy and Sahoo \cite{2008-Ponnusamy-Sahoo} obtained sharp estimate of the pre-Schwarzian norm $\|P_{J[f]}\|$ for functions in the class $\mathcal{C}(A,B)$ for $-1\le B<A\le1$, under certain restrictions on $A$ and $B$ with the help of hypergeometric functions. It is important to note that the Alexander transformation $J[f]\in \mathcal{S}^*(A,B)$ if $f\in\mathcal{C}(A,B)$. In 2023, Ali and Pal \cite{2023-Ali} determined the sharp estimate of the pre-Schwarzian norm for functions in $\mathcal{S}^*(A,B)$ for $-1\le B<A\le1$.\\

Motivated by these above work, in the present article, we find the sharp estimates of the pre-Schwarzian norms for functions in the classes $\mathcal{S}^*_{\lambda e},~\mathcal{C}_{\lambda e}$ with $0<\lambda\le\pi/2$ and $\mathcal{S}^*(q_c),~\mathcal{C}(q_c)$ with $0<c\le1$, respectively.

\begin{thm}\label{R-120}
For $0<\lambda\le\pi/2$, let $f\in\mathcal{S}^*_{\lambda e}$. Then the pre-Schwarzian norm satisfies the following sharp inequality
$$
||P_f||\le \frac{(1-\alpha^2)(e^{\lambda \alpha}+\lambda \alpha-1)}{\alpha},
$$
where $\alpha$ is the unique root in  $(0,1)$ of the equation
\begin{equation}\label{R-35}
1+r^2-2\lambda r^3-e^{\lambda r}(1-\lambda r+r^2+\lambda r^3)=0.
\end{equation}
\end{thm}

For the particular value $\lambda=1$, we get the sharp estimate of the pre-Schwarzian norm for functions in $\mathcal{S}^*(e^z)$.

\begin{cor}\label{R-125}
For any $f\in\mathcal{S}^*(e^z)$, the pre-Schwarzian norm satisfies the following sharp inequality
$$
||P_f||\le \frac{(1-\alpha^2)(e^{\alpha}+\alpha-1)}{\alpha},
$$
where $\alpha$ is the unique root in  $(0,1)$ of the equation
$$1+r^2-2r^3-e^{ r}(1-r+r^2+r^3)=0.$$
\end{cor}

The next theorem, gives the sharp estimate of the pre-Schwarzian norm for functions in the class $\mathcal{S}^*(q_c)$.

\begin{thm}\label{R-130}
For $0<c\le 1$, let $f\in\mathcal{S}^*(q_c)$. Then the pre-Schwarzian norm satisfies the following sharp inequality
$$
||P_f||\le \frac{c(1-\alpha^2)}{2(1-c\alpha)}+\frac{(1-\alpha^2)(1-\sqrt{1-c\alpha})}{\alpha},
$$
where $\alpha$ is the unique root in $(0,1)$ of the equation
\begin{equation}\label{R-80}
-2+4cr-(2+c^2)r^2+2cr^3-c^2r^4+(1-cr)^{3/2}(2-cr+2r^2-3cr^3)=0.
\end{equation}
\end{thm}

For the particular value $c=1$, the class $\mathcal{S}^*(q_c)$ reduce to the class $\mathcal{S}^*(q_1)=:\mathcal{S}^*(\sqrt{1+z})$. The sharp estimate of the pre-Schwarzian norm for functions in $\mathcal{S}^*(\sqrt{1+z})$ is given by the following result.

\begin{cor}\label{R-135}
Let $f\in\mathcal{S}^*(\sqrt{1+z})$ be of the form \eqref{R-01}. Then the pre-Schwarzian norm satisfies the following sharp inequality
$$
||P_f||\le \frac{(1+\alpha)(2-\alpha-2(1-\alpha)^{3/2})}{2\alpha},
$$
where $\alpha$ is the unique root in $(0,1)$ of the equation
$$-2-r^2+(2+r+3r^2)\sqrt{1-r}=0.$$
\end{cor}

The next two theorems, we provide the sharp estimate of the pre-Schwarzian norm of the Alexander transformation for functions in the class $\mathcal{S}^*_{\lambda e}$ and $\mathcal{S}^*(q_c)$.

\begin{thm}\label{R-140}
For any $f\in\mathcal{S}^*_{\lambda e}$ and $g\in\mathcal{C}_{\lambda e}$, the pre-Schwarzian norm satisfies the following sharp inequality
$$
||P_g||=||P_{J[f]}||\le\frac{(1-\alpha^2)(e^{\lambda \alpha}-1)}{\alpha},
$$
where $\alpha$ is the unique root in $(0,1)$ of the equation 
\begin{align}\label{R-116}
\lambda re^{\lambda r}(1-r^2)-(1+r^2)(e^{\lambda r}-1)=0.
\end{align}
\end{thm}

\begin{thm}\label{R-145}
For any $f\in\mathcal{S}^*(q_c)$ and $g\in\mathcal{C}(q_c)$, the pre-Schwarzian norm satisfies the following sharp inequality
$$
||P_g||=||P_{J[f]}||\le \frac{(1-\alpha^2)(1-\sqrt{1-c \alpha})}{\alpha},
$$
where $\alpha$ is the unique root in $(0,1)$ of the equation 
\begin{align}\label{R-119}
2(1+r^2)(1-\sqrt{1-cr})-c(r+3r^3)=0.
\end{align} 
\end{thm}

\section{Proof of main results}
Before we prove our main results, let us discuss two well known results that we will utilize throughout the article to derive our results as well as to construct our extremal functions. Let $\mathcal{B}$ be the class of all analytic functions $\omega:\mathbb{D}\rightarrow\mathbb{D}$ and $\mathcal{B}_0$ be the subclass of $\mathcal{B}$ with $\omega(0)=0$. Functions in $\mathcal{B}_0$ are called Schwarz function. According to Schwarz's lemma, if $\omega\in\mathcal{B}_0$, then $|\omega(z)|\le |z|$ and $|\omega'(0)|\le 1$. The equality occurs in any one of the inequalities if and only if $\omega(z)=e^{i\alpha}z$, $\alpha\in\mathbb{R}$. An extension of Schwarz lemma, known as Schwarz-Pick lemma gives the estimate $|\omega'(z)|\le (1-|\omega(z)|^2)/(1-|z|^2)$ for $z\in\mathbb{D}$ and $\omega\in\mathcal{B}$.

%
%
%

%


\begin{proof}[\textbf{Proof of Theorem \ref{R-125}}]
If $f\in\mathcal{S}^*_{\lambda e},~0<\lambda\le\pi/2$, then
$$
\frac{zf'(z)}{f(z)}\prec e^{\lambda z}.
$$
Thus, there exist a Schwarz function $\omega\in\mathcal{B}_0$ such that
\begin{align*}
\frac{zf'(z)}{f(z)}=e^{\lambda \omega(z)}.
\end{align*}
Taking logarithmic derivative on both sides and a simple calculation gives
\begin{align*}
P_f(z)=\frac{f''(z)}{f'(z)}=\lambda\omega'(z)+\frac{e^{\lambda \omega(z)}-1}{z},
\end{align*}
and so,
\begin{align*}
(1-|z|^2)|P_f(z)|&=(1-|z|^2)\left|\lambda\omega'(z)+\frac{e^{\lambda \omega(z)}-1}{z}\right|\\&\le (1-|z|^2)\left(\lambda|\omega'(z)|+\frac{e^{\lambda| \omega(z)|}-1}{|z|}\right).
\end{align*}
By Schwarz-Pick lemma, we obtain
\begin{align*}
(1-|z|^2)|P_f(z)|\le\lambda(1-|\omega(z)|^2)+\frac{(1-|z|^2)(e^{\lambda |\omega(z)|}-1)}{|z|}.
\end{align*}
For $0\le s:=|\omega(z)|\le|z|<1$, we have
$$(1-|z|^2)|P_f(z)|\le \lambda(1-s^2)+\frac{(1-|z|^2)(e^{\lambda s}-1)}{|z|}.$$
Therefore,
\begin{align}\label{R-50}
||P_f||=\sup\limits_{z\in\mathbb{D}}(1-|z|^2)|P_f(z)|\le\sup\limits_{0\le s\le|z|<1}g(|z|,s),
\end{align}
where $$g(r,s)=\lambda(1-s^2)+\frac{(1-r^2)(e^{\lambda s}-1)}{r}\quad \text{for}~r=|z|.$$
Now we wish to find the supremum of $g(r,s)$ on $\Omega=\{(r,s):0\le r\le s<1\}$.
Clearly,
$$
\frac{\partial g}{\partial r}=-\frac{(1+r^2)(e^{\lambda s}-1)}{r^2}<0.
$$
Therefore, $g(r,s)$ is a strictly decreasing function of $r$ in $[s,1)$ and hence,
\begin{align}\label{R-55}
g(r,s)\le g(s,s)=h(s),
\end{align}
where
\begin{align}\label{R-57}
h(s)=\frac{(1-s^2)(e^{\lambda s}+\lambda s-1)}{s}.
\end{align}
Now, we have to find supremum value of $h(s)$ in $[0,1)$. To do this, we first identify the critical values of $h(s)$ in $(0,1)$. A simple calculation gives
$$
h'(s)=\dfrac{1+s^2-2\lambda s^3-e^{\lambda s}(1-\lambda s+s^2+\lambda s^3)}{s^2}\quad\text{and}\quad h''(s)=\dfrac{k(s)}{s^3},
$$
where
$$
k(s)=-2(1+\lambda s^3)+e^{\lambda s}(2-2\lambda s+\lambda^2s^2-2\lambda s^3-\lambda^2s^4).
$$
Therefore,
$$
k'(s)=-\lambda s^2\left(6+e^{\lambda s}\left(6+6\lambda s-(1-s^2)\lambda^2\right)\right)<0,~\text{for}~s\in(0,1),~\lambda\in(0,\frac{\pi}{2}).
$$
Thus, $k$ is a strictly decreasing function in $(0,1)$. Since $k(0)=0$, this lead us to conclude that $k(s)<0$. Therefore, $h''(s)<0$ for all $s$ in $(0,1)$. Thus, the function $h'$ is a strictly decreasing function in $(0,1)$. In particular, $\lim\limits_{s\rightarrow 0}h'(s)=\lambda^2>0$ and $h'(1)=-2(e^\lambda-1)-2\lambda<0$. This lead us to conclude that the function $h'$ has unique zero, say $\alpha$, in $(0,1)$. Also, the function $h$ is increasing in a neighborhood of $0$ and decreasing in a neighborhood of $1$. This shows that $h$ attains its maximum at $s=\alpha$. Hence, from \eqref{R-50} and \eqref{R-55}, we get the desired result.\\

To show that the estimate is sharp, let us consider the function $f_1$, given by
$$
f_1(z)=z~\mathrm{exp}{\int_0^z\frac{e^{\lambda t}-1}{t}}dt.
$$
The pre-Schwarzian derivative of $f_1$ is given by
$$
P_{f_1}(z)=\frac{e^{\lambda z}+\lambda z-1}{z},
$$
and so,
$$
||P_{f_1}||=\sup\limits_{z\in\mathbb{D}}(1-|z|^2)|P_{f_1}(z)|=\sup\limits_{z\in\mathbb{D}}\frac{(1-|z|^2)|e^{\lambda z}+\lambda z-1|}{z}.
$$
On the positive real axis, we note that
$$
\sup\limits_{0\le r<1}\frac{(1-r^2)(e^{\lambda r}+\lambda r-1)}{r}=\sup\limits_{0\le r<1}h(r)=h(\alpha),
$$
where  $h$ is given by \eqref{R-57} and $\alpha$ is the unique root of \eqref{R-35} in $(0,1)$. Therefore,
$$
||P_{f_1}||=h(\alpha)=\frac{(1-\alpha^2)(e^{\lambda \alpha}+\lambda \alpha-1)}{\alpha}.$$

\end{proof}


Before we prove our next theorem, we prove the following technical lemma.

\begin{lem}\label{R-60}
For a fixed $c$ with $0<c\le 1$, let
\begin{align*}
k_1(s) &=(1-cs)^2 (-8+12cs-3c^2s^2-4cs^3+3c^2s^4),\\
k_2(s) &=-4\sqrt{1-cs}(-2+6cs-6c^2s^2+cs^3+c^3s^3),\\
k(s) &=\frac{k_1(s)+k_2(s)}{(1-cs)^{7/2}},
\end{align*} 
where $0<s<1$. Then $k(s)<0$ for all $s\in(0,1)$ and for each fixed $c\in(0,1]$.
\end{lem}

\begin{proof}
Differentiating $k(s)$ with respect to $s$, one can obtain
\begin{equation}\label{R-65}
k'(s)=\frac{3cs^2\left(k_3(s)+k_4(s)\right)}{2(1-cs)^{\frac{9}{2}}},
\end{equation}
where
$$
k_3(s)=(1-cs)^2(-8+c^2+12cs-5c^2s^2)\quad\text{and}\quad k_4(s)=-8(1-c^2)\sqrt{1-cs}.
$$
A simple calculation gives
\begin{equation}\label{R-70}
k_3'(s)=2c(1-cs)(14-c^2-23cs+10c^2s^2)\quad\text{and}\quad
k_4'(s)=\frac{4c(1-c^2)}{\sqrt{1-cs}}> 0.
\end{equation}
Moreover,
$$k_3''(s)=2c^2(-37+c^2+66cs-30c^2s^2)\quad\text{and}\quad k_3'''(s)=12c^3(11-10cs)>0.$$
This shows that the function $k_3''$ is strictly increasing in $(0,1)$. Therefore, $$k_3''(s)<k_3''(1)=2c^2l(c),$$
where $l(c)=-37+66c-29c^2$ is a strictly increasing function in $(0,1]$ and $l(1)=0$. Thus, $k_3''(s)<0$ for all $s\in (0,1)$. This lead us to conclude that $k_3'$ is a strictly decreasing function in $(0,1)$ and therefore,
$$k_3'(s)>k_3'(1)=2c(1-c)^2(14-9c)> 0.$$
From \eqref{R-70}, we conclude that $k_3'(s)+k_4'(s)>0$ for $s\in (0,1)$ and so, $k_3+k_4$ is strictly increasing in $(0,1)$. Hence,
\begin{equation}\label{R-75}
k_3(s)+k_4(s)<k_3(1)+k_4(1)=l_1(c),
\end{equation}
where
$$l_1(c)=(1-c)^2(-8+12c-4c^2)-8(1-c^2)\sqrt{1-c}.$$
Since
$$l_1'(c)=4(1-c)^2(7-4c)+4(1+5c)\sqrt{1-c}> 0,$$
it follows that, $l_1$ is a strictly increasing function in $(0,1)$ and so, $l_1(c)<l_1(1)=0$. Therefore, from \eqref{R-75}, we have
$$k_3(s)+k_4(s)<k_3(1)+k_4(1)=l_1(c)<l_1(1)=0\quad\text{for all}~s\in (0,1).$$
Hence, from \eqref{R-65}, we conclude that the function $k$ is strictly decreasing in $(0,1)$ and therefore, $k(s)<k(0)=0$ for all $s\in (0,1)$ and $c\in(0,1]$.
\end{proof}


\begin{proof}[\textbf{Proof of Theorem \ref{R-130}}]
If $f\in\mathcal{S}^*(q_c),~0<c\le 1$, then
$$
\frac{zf'(z)}{f(z)}\prec \sqrt{1+cz}.
$$
Thus, there a exist Schwarz function $\omega(z)\in\mathcal{B}_0$ such that
\begin{align*}
\frac{zf'(z)}{f(z)}=\sqrt{1+c\omega(z)}.
\end{align*}
Taking logarithmic derivative on both sides and a simple calculation gives
\begin{equation}\label{R-85}
P_f(z)=\frac{f''(z)}{f'(z)}=\frac{c\omega'(z)}{2\big(1+c\omega(z)\big)}-\frac{1-\sqrt{1+c\omega(z)}}{z},
\end{equation}
and so,
\begin{align}\label{R-90}
(1-|z|^2)|P_f(z)|&\le(1-|z|^2)\left(\frac{|c\omega'(z)|}{2|1+c\omega(z)|}+\frac{|1-\sqrt{1+c\omega(z)}|}{|z|}\right)\\&\le(1-|z|^2)\left(\frac{c|\omega'(z)|}{2\left(1-c|\omega(z)|\right)}+\frac{\left(1-\sqrt{1-c|\omega(z)|}\right)}{|z|}\right)\nonumber.
\end{align}
By Schwarz-Pick lemma, we have
$$(1-|z|^2)|P_f(z)|\le \frac{c(1-|\omega(z)|^2)}{2(1-c|\omega(z)|)}+\frac{(1-|z|^2)(1-\sqrt{1-c|\omega(z)}|)}{|z|}.$$
For $0\le s:=|\omega(z)|\le|z|<1$, we obtain
$$(1-|z|^2)|P_f(z)|\le \frac{c(1-s^2)}{2(1-cs)}+\frac{(1-|z|^2)(1-\sqrt{1-cs})}{|z|}.$$
Therefore,
\begin{align}\label{R-95}
||P_f||=\sup\limits_{z\in\mathbb{D}}(1-|z|^2)|P_f(z)|\le \sup\limits_{0\le s\le|z|<1}g(|z|,s),
\end{align}
where $g$ is given by
\begin{align}\label{R-100}
g(r,s)=\frac{c(1-s^2)}{2(1-cs)}+\frac{(1-r^2)(1-\sqrt{1-cs})}{r}\quad \text{for}~r=|z|.
\end{align}
Now we wish to find the supremum of $g(r,s)$ on $\Omega=\{(r,s):0\le r\le s<1\}$. To do this we consider two different cases.\\

\textbf{Case-1}: Let $c=1$. Then
$$
g(r,s)=\frac{1+s}{2}+\frac{(1-r^2)(1-\sqrt{1-s})}{r}.
$$
Taking partial derivative with respect to $s$, we get
$$
\frac{\partial g}{\partial s}=\frac{1}{2}+\frac{1-r^2}{2r\sqrt{1-s}}>0.
$$
This shows that the function $g$ is a strictly increasing of $s$ in $[0,r]$ and therefore,
\begin{align}\label{R-102}
g(r,s)\le g(r,r)=g_1(r),
\end{align}
where
\begin{align}\label{R-104}
g_1(r)=\frac{(1+r)(2-r-2(1-r)^{3/2})}{2r}.
\end{align}
 Now, we have to find maximum value of $g_1$ in $(0,1)$. To determine the maximum value, we wish to find the critical value of $g_1$ in $(0,1)$. A simple calculation gives
$$
g_1'(r)=\frac{-2-r^2+(2+r+3r^2)\sqrt{1-r}}{2r^2}\quad\text{and}\quad g_1''(r)=\frac{g_2(r)}{4r^3\sqrt{1-r}},
$$
where
\begin{align*}
g_2(r)&=-8+4r+r^2-3r^3+8\sqrt{1-r}\\&
=8\left(\sqrt{1-r}-1+\frac{r}{2}+\frac{r^2}{8}\right)-3r^3\\&=-8\left(\frac{3r^3}{48}+\frac{15r^4}{384}+\cdots\right)-3r^3.
\end{align*}
From the Taylor series expansion it is clear that $g_1''(r)<0$ for all $r\in (0,1)$. Therefore, the function $g_1'$ is a strictly decreasing in $(0,1)$. In particular, $\lim_{r\rightarrow 0}g_1'(r)=5/8>0$ and $g_1'(1)=-3/2<0$.  This lead us to conclude that the function $g_1'$ has unique zero, say $\alpha$, in $(0,1)$. Note that $\alpha$ is the unique zero in $(0,1)$ of the equation \eqref{R-80} for $c=1$. Moreover, the function $g_1$ is increasing in a neighborhood of $0$ and decreasing in a neighborhood of $1$. Thus, the function $g_1$ has maximum at $\alpha$. Hence from \eqref{R-95} and \eqref{R-102}, we get the desired result.\\

\textbf{Case-2}: Let $c\neq1$. Then differentiating \eqref{R-100} with respect to $r$, we get
$$
\frac{\partial g}{\partial r}=-\frac{(1+r^2)(1-\sqrt{1-cs})}{r}<0.
$$
Therefore, $g(r,s)$ is a strictly decreasing function of $r$ in $[s,1)$. Thus,
\begin{align}\label{R-105}
g(r,s)\le g(s,s)=h(s),
\end{align}
where
\begin{align}\label{R-110}
h(s)=\frac{c(1-s^2)}{2(1-cs)}+\frac{(1-s^2)(1-\sqrt{1-cs})}{s}.
\end{align}
Now we wish to find maximum value of $h(s)$ for $0\le s<1$. To do this, we have to find critical points of $h(s)$ in $(0,1)$. A simple but tedius computation gives
$$
h'(s)=\frac{h_1(s)+h_2(s)}{2s^2(1-cs)^2}\quad\text{and}\quad h''(s)=\frac{k(s)}{4s^3},
$$
where
\begin{align}\label{R-115}
\begin{cases}
h_1(s)&=-2+4cs-(2+c^2)s^2+2cs^3-c^2s^4,\\
h_2(s)&=(1-cs)^{3/2}(2-cs+2s^2-3cs^3),
\end{cases}
\end{align}
and $k(s)$ is given in Lemma \ref{R-60}. By Lemma \ref{R-60}, $k(s)<0$ for $s\in(0,1)$ and so $h''(s)<0$ for $s\in (0,1)$. Hence, the function $h'$ is strictly decreasing in $(0,1)$ and
$$\lim_{s\rightarrow 0}h'(s)=\frac{5c^2}{8}>0\quad\text{and}\quad h'(1)=-2(1-\sqrt{1-c})-\frac{c}{1-c}<0.$$
This lead us to conclude that the function $h'$ has unique zero, say $\alpha$  in $(0,1)$. Note that $\alpha$ is the unique zero in $(0,1)$ of the equation \eqref{R-80} for $c\neq1$.  Moreover, the function $h$ is increasing in a neighborhood of $0$ and decreasing in a neighborhood of $1$. This shows that $h$ attains its maximum at $s=\alpha$. Hence from \eqref{R-95} and \eqref{R-105}, we get the desired result.\\

To show that the estimate is sharp, let us consider the function $f_2$ defined by
$$
f_2(z)=z~\mathrm{exp}{\int_0^z\frac{\sqrt{1-ct}-1}{t}dt}.
$$
A simple calculation gives
$$
P_{f_2}(z)=-\frac{c}{2(1-cz)}-\frac{1-\sqrt{1-cz}}{z}.
$$
and
$$
||P_{f_2}||=\sup\limits_{z\in\mathbb{D}}(1-|z|^2)|P_{f_2}(z)|=\sup\limits_{z\in\mathbb{D}}\left|\frac{c(1-|z|^2)}{2(1-cz)}+\frac{(1-|z|^2)(1-\sqrt{1-cz})}{z}\right|.
$$
On the positive real axis, we have
\begin{align*}
\sup\limits_{0\le r<1}\left(\frac{c(1-r^2)}{2(1-cr)}+\frac{(1-r^2)(1-\sqrt{1-cr})}{r}\right)=
\begin{cases}
\sup\limits_{0\le r<1}h(r)=h(\alpha),\quad
\text{for}\quad c\neq1,\\\
\sup\limits_{0\le r<1}g_1(r)=g_1(\alpha),\quad
\text{for}\quad c=1,
\end{cases}
\end{align*}
where $h$ and $g_1$ are given by \eqref{R-110} and \eqref{R-104}, respectively and $\alpha$ is the unique zero in $(0,1)$ of \eqref{R-80}. Thus,
$$
||P_{f_2}||=\frac{c(1-\alpha^2)}{2(1-c\alpha)}+\frac{(1-\alpha^2)(1-\sqrt{1-c\alpha})}{\alpha}.
$$
This completes the proof.
\end{proof}


\begin{proof}[\textbf{Proof of Theorem \ref{R-140}}]
Let $f\in\mathcal{S}^*_{\lambda e}$. Then $g=J[f]$ is in the class $\mathcal{C}_{\lambda e}$ and so 
$$
1+\frac{zg''(z)}{g'(z)}\prec e^{\lambda z}.
$$
Thus, there exist a function $\omega(z)\in\mathcal{B}_0$ such that
$$
1+\frac{zg''(z)}{g'(z)}= e^{\lambda \omega(z)}.
$$
By a simple calculation, we have
\begin{align*}
(1-|z|^2)|P_g(z)|&=(1-|z|^2)\left|\frac{g''(z)}{g'(z)}\right|\\&=(1-|z|^2)\left|\frac{e^{\lambda \omega(z)}-1}{z}\right|\\&\le(1-|z|^2)\left(\frac{e^{\lambda |\omega(z)|}-1}{|z|}\right)\\&\le(1-|z|^2)\left(\frac{e^{\lambda |z|}-1}{|z|}\right)
\end{align*}
and so,
\begin{align}\label{R-117}
||P_g||=\sup\limits_{z\in\mathbb{D}}(1-|z|^2)|P_g(z)|\le \sup\limits_{0\le |z|<1} h(|z|),
\end{align}
where $h$ is given by 
\begin{align}\label{R-118}
h(r)=\frac{(1-r^2)\left(e^{\lambda r}-1\right)}{r}\quad \mathrm{for}~r=|z|.
\end{align}
 A simple calculation gives
\begin{align*}
h'(r)=\frac{\lambda re^{\lambda r}(1-r^2)-(1+r^2)(e^{\lambda r}-1)}{r^2}\quad \mathrm{and}\quad
h''(r)=-\frac{k(r)}{r^3},
\end{align*}
where $$
k(r)=2(1+e^{\lambda r})+\lambda r e^{\lambda r}(2+2r^2+\lambda r^3-\lambda r)>0
$$ for $0<r<1$ and $0< \lambda\le\pi/2$. Thus we have $h''(r)<0$. Also, we have $h'(1)=-2(e^\lambda-1)<0$ and $h'(0)=\lambda^2/2>0$ which shows that $h'(r)$ has a unique zero say $\alpha$, in $(0,1)$. Thus $h$ attains its maximum value at $r=\alpha$. Therefore, from \eqref{R-118} we get
$$
||P_f||\le h(\alpha).
$$

To show that the estimate is sharp, let us consider the function $f_3$, given by
$$
f_3(z)=\int_0^z\left(\mathrm{exp}{\int_0^u\frac{e^{\lambda t}-1}{t}}dt\right)du.
$$
The pre-Schwarzian norm of $f_3$ is given by
$$
||P_{f_3}||=\sup\limits_{z\in\mathbb{D}}(1-|z|^2)|P_{f_3}(z)|=\sup\limits_{z\in\mathbb{D}}\frac{(1-|z|^2)|e^{\lambda z}-1|}{|z|}.
$$
On the positive real axis, we note that
$$
\sup\limits_{0\le r<1}\frac{(1-r^2)(e^{\lambda r}-1)}{r}=\sup\limits_{0\le r<1}h(r)=h(\alpha),
$$
where  $h$ is given by \eqref{R-118} and $\alpha$ is the unique root of \eqref{R-116} in $(0,1)$. Therefore,
$$
||P_{f_3}||=\frac{(1-\alpha^2)(e^{\lambda \alpha}-1)}{\alpha}.$$
\end{proof}

\begin{proof}[\textbf{Proof of Theorem \ref{R-145}}]
Let $f\in\mathcal{S}^*(q_c)$. Then $g=J[f]$ is in the class $\mathcal{C}(q_c)$ and so,
$$
1+\frac{zg''(z)}{g'(z)}\prec \sqrt{1+cz}.
$$
Thus, there exist a function $\omega(z)\in\mathcal{B}_0$ such that
$$
1+\frac{zg''(z)}{g'(z)}= \sqrt{1+c\omega(z)}.
$$
By a simple calculation we have
\begin{align*}
(1-|z|^2)|P_g|&=(1-|z|^2)\left|\frac{g''(z)}{g'(z)}\right|\\&=(1-|z|^2)\left|\frac{1-\sqrt{1+c\omega(z)}}{z}\right|\\&\le(1-|z|^2)\left(\frac{\left(1-\sqrt{1-c|\omega(z)|}\right)}{|z|}\right)\\&\le\frac{(1-|z|^2)\left(1-\sqrt{1-c|z|}\right)}{|z|}
\end{align*}
and so,
\begin{align}\label{R-120}
||P_g||=\sup\limits_{z\in\mathbb{D}}(1-|z|^2)|P_g(z)|\le \sup\limits_{0\le s\le|z|<1} h(|z|),
\end{align}
where $h$ is given by 
$$
h(r)=\frac{(1-r^2)\left(1-\sqrt{1-cr}\right)}{r}\quad \mathrm{for}~r=|z|.
$$
A simple calculation gives
$$
h'(r)=\frac{2(1+r^2)(1-\sqrt{1-cr})-c(r+3r^3)}{2r^2\sqrt{1-cr}}\quad \mathrm{and} \quad h''(r)=\frac{k(r)}{4r^3},
$$
where
$$
k(r)=\frac{-8+12cr-3c^2r^2-4cr^3+3c^2r^4+8(1-cr)^{\frac{3}{2}}}{(1-cr)^{\frac{3}{2}}}.
$$
Further
$$
k'(r)=\frac{3cr^2k_1(r)}{2(1-cr)^{\frac{5}{2}}},
$$
where
$k_1(r)=-8+c^2+12cr-5c^2r^2$ is a increasing function  for $0<r<1$. Thus we have $k_1(r)<k_1(1)=4(1-c)(c-2)<0$ and consequently
$k'(r)<0$. Thus $k(r)$ is a decreasing function in $r$ and so, $k(r)<k(0)=0.$ Therefore 
$h''(r)<0$ and so $h'$ is a decreasing function in $r$. Also, we have $h'(1)=-2(1-\sqrt{1-c})<0$ and $h'(0)=\frac{c^2}{8}>0$ which shows that $h'(r)$ has a unique zero, say $\alpha$ in $(0,1)$. Thus $h$ attains its maximum value at $r=\alpha$. Therefore, from \eqref{R-120} we get
$$
||P_f||\le h(\alpha).
$$

To show that the estimate is sharp, let us consider the function $f_4$ defined by
$$
f_4(z)=\int_0^z\left(\mathrm{exp}{\int_0^u\frac{\sqrt{1-ct}-1}{t}}dt\right)du.
$$
A simple calculation gives
$$
||P_{f_4}||=\sup\limits_{z\in\mathbb{D}}(1-|z|^2)|P_{f_4}(z)|=\sup\limits_{z\in\mathbb{D}}\left|\frac{(1-|z|^2)(1-\sqrt{1-cz})}{z}\right|.
$$
On the positive real axis, we have
\begin{align*}
\sup\limits_{0\le r<1}\left(\frac{(1-r^2)(1-\sqrt{1-cr})}{r}\right)=
\sup\limits_{0\le r<1}h(r)=h(\alpha),
\end{align*}
where $h$ is given by \eqref{R-120} and $\alpha$ is the unique zero in $(0,1)$ of \eqref{R-119}. Thus,
$$
||P_{f_4}||=\frac{(1-\alpha^2)(1-\sqrt{1-c\alpha})}{\alpha}.
$$

\end{proof}

\vspace{4mm}
\noindent\textbf{Data availability:}
Data sharing not applicable to this article as no data sets were generated or analyzed during the current study.\vspace{4mm}
\\
\noindent\textbf{Authors Contributions:}
All authors contributed equally to the investigation of the problem and the order of the authors is given alphabetically according to their surnames. All authors read and approved the final manuscript.\vspace{4mm}
\\


\begin{thebibliography}{100}

\bibitem{2014-Aghalary}
{\sc R. Aghalary and Z. Orouji}, Norm Estimates of the Pre-Schwarzian Derivatives for $\alpha$-Spiral-Like Functions of Order $\rho$, \textit{Complex Anal. Oper. Theory} \textbf{8}(4), 791--801 (2014).

\bibitem{2023-Ali}
{\sc M. F. Ali and S. Pal}, Pre-Schwarzian norm estimates for the class of Janowski starlike functions, \textit{Monatsh. Math.} \textbf{201}, no. 2, 311--327 (2023).

\bibitem{2011-Aouf-Dziok-Sokol}
{\sc M. K. Aouf, J. Dziok and J. Sok\'{o}\l}, On a subclass of strongly starlike functions, {\it Appl. Math. Lett.} {\bf 24},
27--32 (2011).

\bibitem{1972-Becker}
{\sc J. Becker}, L\"{o}wnersche Differentialgleichung und quasikonform fortsetzbare schlichte Funktionen, {\it J. Reine Angew. Math.} {\bf255}, 23--43 (1972).


\bibitem{1984-Becker-Pommerenke-1984}
{\sc J. Becker and C. Pommerenke}, Schlichtheitskriterien und Jordangebiete, {\it J. Reine Angew. Math.} {\bf 354}, 74--94 (1984).

\bibitem{2019-Kumar}
{\sc N. E. Cho, V. Kumar, S. S. Kumar and V. Ravichandran}, Radius problems for starlike functions associated with the sine function, \textit{Bull. Iranian Math. Soc.} \textbf{45}(1), 213--232 (2019).


\bibitem{2005-Choi}
{\sc J. H. Choi}, {\sc Y. C. Kim}, {\sc  S. Ponnusamy} and {\sc  T.Sugawa}, Norm estimates for the Alexander transforms of
convex functions of order alpha. \textit{J. Math. Anal. Appl.} \textbf{303}, 661--668 (2005).

\bibitem{1983-Duren}
{\sc P. L. Duren},  Univalent functions (Grundlehren der
mathematischen Wissenschaften 259, New York, Berlin, Heidelberg, Tokyo, Springer-Verlag) (1983).


\bibitem{2020-Goel}
{\sc P. Goel and S. S. Kumar}, Certain class of starlike functions associated with modified sigmoid function, \textit{Bull. Malays. Math. Sci. Soc.} \textbf{43}, 957--991 (2020).

\bibitem{1983-Goodman}
{\sc A. W. Goodman}, Univalent Functions, Vol. {\bf I} and {\bf II} (Mariner Publishing Co., Tampa, Florida)(1983).

\bibitem{1973-Janowski}
{\sc W. Janowski}, Extremal problems for a family of functions with positive real part and for some related families, \textit{Ann. Polon. Math.} \textbf{23}, 159--177 (1973).




647--657 (2000).

\bibitem{2004-Kim-Ponnusamy-Sugawa}
{\sc Y. C. Kim, S. Ponnusamy and T. Sugawa}, Mapping properties of nonlinear integral operators and pre-Schwarzian derivatives, \textit{J. Math. Anal. Appl.} \textbf{299}, 433--447 (2004).

\bibitem{2002-Kim-Sugawa}
{\sc Y. C. Kim and T. Sugawa}, Growth and coefficient estimates for uniformly locally univalent functions on the unit disk, {\it Rocky Mountain J. Math.} {\bf 32}, 179--200 (2002).

\bibitem{2006-Kim-Sugawa}
{\sc Y. C. Kim and T. Sugawa}, Norm estimates of the pre-Schwarzian derivatives for certain classes of univalent
functions, \textit{Proc. Edinb. Math. Soc.}(2) \textbf{49}(1), 131--143 (2006).

\bibitem{2018-Kumar-Ravichandran}
{\sc S. Kumar and V. Ravichandran}, Subordinations for functions with positive real part, \textit{Complex Anal. Oper. Theory} \textbf{12}(5), 1179--1191 (2018).

\bibitem{2017-KUMAR-RAVICHANDRAN-VERMA}
{\sc S. Kumar, V. Ravichandran AND S. Verma}, Initial coefficients of starlike functions with real coefficients, \textit{Bull. Iranian Math. Soc.} \textbf{43}, 1837--1854 (2017).

\bibitem{1992-Ma-Minda}
{\sc W. Ma and D. A. Minda}, Unified treatment of some special classes of univalent functions, In: Proceedings of the Conference on Complex Analysis, Tianjin. Conf. Proc. Lecture Notes Anal., I, Int. Press,
Cambridge, MA, pp. 157--169 (1992).

\bibitem{2014-Mendiratta}
{\sc R. Mendiratta, S. Nagpal and V. Ravichandran}, On a subclass of strongly starlike functions associated with exponential function, {\textit Bull. Malays. Math. Sci. Soc.} \textbf{38}(1), 365-386 (2014).

\bibitem{2000-Okuyama}
{\sc  Y. Okuyama}, The norm estimates of pre-Schwarzian derivatives of spiral-like functions,  \textit{Complex Var. Theory Appl.} \textbf{42}, 225-239 (2000).


\bibitem{2008-Parvatham-Ponnusamy-Sahoo}
{\sc  R. Parvatham, S. Ponnusamy, and S. K. Sahoo}, Norm estimates for the Bernardi integral transforms of functions defined by subordination. \textit{Hiroshima Math. J.} \textbf{38}, 19--29 (2008).

\bibitem{2008-Ponnusamy-Sahoo}
{\sc S. Ponnusamy and  S. K. Sahoo}, Norm estimates for convolution transforms of certain classes of analytic functions, \textit{J. Math. Anal. Appl.} \textbf{342}, 171--180 (2008).

\bibitem{2010-Ponnusamy}
{\sc S. Ponnusamy and  S. K. Sahoo}, Pre-Schwarzian norm estimates of functions for a subclass of strongly starlike functions, \textit{Mathematica} \textbf{52}(75), 47--53 (2010).

\bibitem{2015-Raina}
{\sc R. K. Raina and J. Sok\'{o}\l}, Some properties related to a certain class of starlike functions, \textit{C. R. Math. Acad. Sci. Paris} \textbf{353}(11), 973--978 (2015).

\bibitem{1936-Robertson}
{\sc M. I. S. Robertson}, On the theory of univalent functions, \textit{Ann. Math.} \textbf{37}(2), 374--408 (1936).

\bibitem{1993-Ronning}
{\sc F. R\o nning}: Uniformly convex functions and a corresponding class of starlike functions, \textit{Proc. Am. Math. Soc.} \textbf{118}(1), 189--196 (1993).

\bibitem{2020-L.Shi}
{\sc L. Shi, Z. G. Wang, R. L. Su and  M. Arif}, Initial successive coefficients for certain classes of univalent functions involving the exponential function, {\it J. Math. Inequal.} {\bf 14}(2), 1183--1201 (2020).

\bibitem{1996-Sokol}
{\sc J. Sok\'{o}\l~and  J. Stankiewicz}, Radius of convexity of some subclasses of strongly starlike functions, \textit{Zeszyty Nauk. Politech. Rzeszowskiej Mat.} {\textbf 19}, 101--105 (1996).


\bibitem{1971-Stankiewicz}
{\sc J. Stankiewicz}, Quelques probl\`{e}mes extr\'{e}maux dans les classes des fonctions $\alpha$-angulairement \'{e}toil\'{e}es, \textit{Ann. Univ. Mariae Curie-Sk\l odowska Sect. A.} \textbf{20}(1966), 59--75 (1971).

\bibitem{1998-Sugawa}
{\sc T. Sugawa}, On the norm of pre-Schwarzian derivatives of strongly starlike functions, \textit{Ann. Univ. Mariae Curie-Sk\l odowska Sect. A} \textbf{52}(2), 149--157 (1998).


\bibitem{1976-Yamashita}
{\sc S. Yamashita}, Almost locally univalent functions, \textit{Monatsh. Math.} \textbf{81}, 235--240 (1976).


\bibitem{1999-Yamashita}
{\sc S. Yamashita}, Norm estimates for function starlike or convex of order alpha, {\it Hokkaido Math. J.} {\bf 28}(1), 217--230 (1999).
\end{thebibliography}
\end{document}